\documentclass[review]{elsarticle}
\usepackage{amsmath}
\usepackage[utf8]{inputenc}
\usepackage{amsfonts}
\usepackage{amsthm}
\usepackage{lineno,hyperref}
\modulolinenumbers[5]

\journal{Journal of \LaTeX\ Templates}
\newtheorem{theorem}{\indent Theorem}
\newtheorem{lemma}{\indent Lemma}
\newtheorem{definition}{\indent Definition}
\newtheorem{proposition}{\indent Proposition}
\begin{document}

\begin{frontmatter}

\title{Towards the standard Majorana representations of 3-transposition groups}


\author[add]{Albert Gevorgyan\corref{mycorrespondingauthor}\footnote{Research supported by Imperial College London President's PhD Scholarships}}
\cortext[mycorrespondingauthor]{Corresponding author}
\ead{a.gevorgyan20@imperial.ac.uk}

\address[add]{Department of Mathematics, Imperial College, 180 Queen's Gt., London, SW7 2AZ, UK}

\begin{abstract}
    In this paper, we discuss $3$-transposition groups. In particular, we find sizes of maximal symmetric subgroups of the groups, which are in Fischer list. In addition, we build faithful representations of symmetric groups in orthogonal, symplectic, and unitary spaces, which fully verify the main result. Thus, we find candidate groups, for which one may be able to build standard Majorana representations.  
\end{abstract}


\end{frontmatter}

\section{Introduction}
Assume that $G$ is a group generated by a set of its involutions $T$, which is a union of conjugacy classes of $T$. In addition, let $W$ be a linear representation of $G$, which is also equipped with a positive definite form and an algebra product.
\begin{definition} Let the quintet
\begin{equation*}
    (G,T,W,\varphi,\psi)
\end{equation*}
be such that $\varphi:G \rightarrow Aut(W)$ is faithful, $\psi:T \rightarrow W$ is injective with $(t^{\psi})^{g^\psi}=(t^g)^{\psi}$ for all $g \in G, t \in T$. If the quintet satisfies the axioms (M1)-(M8) described in \cite{Iv10}, then it is called a Majorana representation of $G$.    
\end{definition}
From the definition one may deduce that if $T' \subseteq T$ and $H$ is the subgroup of $G$ generated by the involutions in $T'$, then $(H,T',W,\varphi|_H,\psi|_{T'})$ is a Majorana representation of $H$. \newline \newline
In \cite{Iv09} the author shows that if the group $G$ admits a Majorana representation, then the order of the product of any pair of involutions from $T$ is at most $6$. Later in \cite{Iv10} the authors build all the existing Majorana representations of the group $S_4$. Moreover, in \cite{Iv11} and \cite{Se11} authors build Majorana representations of the groups $A_5$, $A_6$, and $A_7$, where there the set $T$ of the involutions is the set of bitranspositions (the set of elements of $A_n$ of cycle type $2^2$). \newline \newline
In \cite{Fr16}, the authors discuss whether the group $A_{13}$ admits a standard Majorana representation. The authors assume that such a representation $V$ exists and under that assumption show that the form defined on $V$ is not positive definite. Thus, they deduce that if $A_n$ admits a standard Majorana representation, then $n \le 12$.
\begin{definition} (\cite{As96}) Let $n \ge 3$ be an integer. We call the pair $(G,D)$, where $G$ is a group and $D$ is a set of involutions of $G$ stable under conjugation, an $n$-transposition group, if the set $D$ generates $G$ and $o(d_1d_2) \le n$ for all $d_1,d_2 \in D$.
\end{definition}
Sometimes, when it is obvious from the context to which involutions $D$ we refer, by abuse of terminology, we say that "$G$ is an $n$-transpostion group". In \cite{Iv09} it is shown that the Monster group $M$ is generated by a conjugacy class of $2A$-involutions. It is also known that the order of the product of any two $2A$-involutions is at most $6$. Therefore, we obtain that $(M,2A)$ is a $6$-transposition group. Furthermore, there exists a Majorana representation $(M,2A,V_M,\varphi,\psi)$ of the Monster group $M$. \newline \newline
The main objects of interest of this paper are the $3$-transposition groups. Our main goal is to find groups $G$, which are candidates to admit Majorana representations.
\begin{theorem} \label{Nort} Suppose $(G,D)$ is a finite $3$-transposition group.
Let $S=\{p \cdot q | p \ne q \in D, [p,q]=1\}$ and $H$ be the subgroup generated by the elements in $S$. Then, $(H,S)$ is a  $6$-transposition group.
\end{theorem}
This theorem is proven in \cite{No17} and is an important starting point for this paper. In particular, it is known that some $3$-transposition groups are embedded into the Monster, so that the images of the elements of $S$ are $2A$-involutions. In addition, if such an embedding does not exist, there may still exist a Majorana representation for these groups.
\begin{theorem} (Fischer, \cite{As96,Fi69,Fi71})
Assume that $(G,D)$ is a $3$-transposition group so that $Z(G)$ is trivial and $[G,G]$ is simple. Then one of the following holds:
\begin{itemize}
    \item $G \cong S_n$, $D$ is the set of all transpositions;
    \item $G \cong Sp_n(2)$, $D$ is the set of transvections $t(v)$, where $v \ne 0$;
    \item $G \cong U_n(2)$, $D$ is the set of transvections $t(v)$, where $v \ne 0$ is singular;
    \item $G \cong O_n^{\epsilon}(2)$, $D$ is the set of transvections $t(v)$, where $v$ is non-singular;
    \item $G \cong PO_n^{\mu,\pi}(3)$, $D$ is the conjugacy class of reflections $t(v)$, where $t(v)=\pi$;
    \item $G$ is of type $M(22),M(23)$ or $M(24)$ and $D$ is a uniquely determined class of transpositions of $G$;
\end{itemize}
\end{theorem}
In this paper for each of above groups $(G,D)$ we find the maximal number $n$ , so that there exists a monomorphism $\varphi : S_n \rightarrow G$ with $\varphi(i j) \in D$, for all transpositions $(i j)$. If we have embedding $\varphi:S_n \rightarrow G$, then we may restrict it to $\varphi':A_n \rightarrow G$, and images of bitranspositions are the products of commuting transpositions of $D$, which form a set of $6$-transpositions by the theorem \ref{Nort}. Thus, if $G$ admits a standard Majorana representation, then so does $A_n$. Therefore, $n \le 12$. We will show that the classification of the candidate groups is the following:
\begin{theorem} \label{main} Assume that $G$ is a group from the Fischer's list such that $\varphi(G) \le 12$. Then one of the following holds:
\begin{itemize}
    \item $G \cong S_n$, with $n \le 12$;
    \item $G \cong Sp_n(2)$, with $n \le 10$;
    \item $G \cong U_n(2)$, with $n \le 11$;
    \item $G \cong {O_n}^{+,+}(3)$, with $n \le 11$;
    \item $G \cong {O_n}^{-,+}(3)$, with $n \le 12$;
    \item $G \cong {O_n}^{\epsilon}(2)$, with $n \le 11$;
    \item $G$ is one of the Fischer groups $M(22)$, $M(23)$ and $M(24)$.
\end{itemize}
\end{theorem}
In this paper, we adopt an elementary combinatorial approach for finding the maximal size of symmetric subgroups of these groups. We do not consider any character tables and all the calculations are carried out by hand. In addition, we build geometrical representations of these groups, which verify the main result for the classical groups.
\section{Notation}
In this paper we use the following notation for the classical groups, which agrees with the notation of \cite{As96}:
\begin{itemize}
    \item For an even integer $n$ we write $Sp_n(2)$ for the isometry group of the $n$-dimensional symplectic space $V$ over the $2$-element field. This group is generated by transvenctions $t(v):w \rightarrow w+(v,w)v$, where $v$ is a non-zero vector.
    \item For an integer $n$ we write $U_n(2)$ for the isometry group of the $n$-dimensional unitary space $V$ over the $4$-element field. This group is generated by transvenctions $t(v):w \rightarrow w+(v,w)v$, where $v$ is a non-zero singular vector.
    \item For an integer $n$ we write $O^{\epsilon}_n(2)$ ($\epsilon \in \{+,-\}$) for the isometry group of the $n$-dimensional orthogonal space $V$ of type $\epsilon$ over the $2$-element field. This group is generated by transvenctions $t(v):w \rightarrow w+(v,w)v$, where $v$ is a non-singular vector.
    \item For an integer $n$ we write $O^{\mu,\pi}_n(3)$ ($\mu, \pi \in \{+,-\}$) for the subgroup of the isometry group of the $n$-dimensional orthogonal space $V$ over the $3$-element field, a Gram matrix of a basis of which has determinant equal to $\mu$,  generated by the reflections $t(v)$ for vectors with $Q(v)=\pi$.
\end{itemize}
In each case $t(v)$ and $t(u)$ commute if and only if $(v,u)=0$. \newline \newline
It also known that, when $n$ is odd, then $O_n^{+}(2) \cong O_n^{-}(2) \cong Sp_{n-1}(2)$. Furthermore, $O_n^{\mu,\pi}(3) \cong O_n^{\mu,-\pi}(3)$, if $n$ is even, and $O_n^{\mu,\pi}(3) \cong O_n^{-\mu,-\pi}(3)$, if $n$ is odd. \newline \newline
In what follows, $\varphi(G)$ denotes the maximal number $n$ , so that there exists a monomorphism $\varphi : S_n \rightarrow G$ with $\varphi(i j) \in D$, for all transpositions $(i j)$.
\section{Auxiliary results}
In this section, we prove several lemmas, which later help us to prove the results of the main part.
\begin{definition} Let $(G,D)$ be a $3$-transposition group. We define a chain of transpositions as a sequence of pairwise distinct transpositions $g_1,...,g_m$, for which $[g_i,g_j]=1$ (with $i \ne j$) if and only if $|i-j| \ne 1$. We call $m$ the length of the chain. 
\end{definition}
Recall that, when $G$ is an orthogonal, symplectic or unitary group, then the transpositions are of form $t(v)$ for some vectors $v$ of the corresponding vector space, and the transpositions commute if and only if the corresponding vectors are orthogonal. Thus, while working with such groups, if the transpositions $g_1=t(v_1),...,g_m=t(v_m)$ form a chain, we say that the vectors $v_1,...,v_m$ form a chain of length $m$. Note that if $v_1,...,v_m$ form a chain, then no pair of vectors may be collinear. 
\begin{lemma} (big) Let $(G,D)$ be a $3$-transposition group. Let $m$ be the length of the maximal chain of transpositions of $G$. Then, $\varphi(G)=m+1$.
\end{lemma}
\begin{proof}
First, let us note that if there exists a monomorphism $\psi : S_n \rightarrow G$, then the transpositions $d_1=\psi(1, 2)$, $d_2=\psi(2, 3)$, ..., $d_{n-1}=\psi(n-1, n)$ form a chain. Hence, $m \ge \varphi(G)-1$, so $\varphi(G) \le m+1$. \newline \newline
On the other hand, suppose there exist transpositions $d_1,...,d_m$ forming a chain. Then, if we create the non-commuting graph of $\{ d_1, ..., d_m \}$, it is a Coxeter graph of form $A_m$. It is well-known fact that the corresponding Coxeter group is isomorphic to $S_{m+1}$. Thus, there exists a homomorphism $\psi:S_{m+1} \rightarrow <d_1,...,d_m>$ with $\psi(i, i+1)=d_i$. \newline \newline
Note that $ker(f) \unlhd S_{m+1}$. Moreover, the normal subgroups of $S_{m+1}$ are  $\{ 1 \}$, $A_{m+1}$ and $S_{m+1}$ and possibly the Klein $4$-group (if $m+1=4$). In addition, if $ker(f) \ne \{1\}$, then the factor group $S_{m+1} / ker(f)$ does not contain $m$ involutions. Therefore, $ker(f)$ is trivial and $f$ is an isomorphism. Thus $\varphi(G) \ge m+1$.
\end{proof}
In addition, we use the following classical result, which can be found in \cite{As96}.
\begin{lemma} (Witt):
Let $V$ be an orthogonal, symplectic or unitary space. Let $U$ and $W$ be subspaces of $V$ and suppose that $\alpha : U \rightarrow W$ is an isometry of subspaces. Then, $\alpha$ extends to an isometry of $V$.
\end{lemma}
\begin{lemma} (small):
Let $V$ be an orthogonal, symplectic or unitary space. Let the vectors $v_1,...,v_{n+1} \in V$ form a chain. Then the vectors $v_1,...,v_n$ are linearly independent.
\end{lemma}
\begin{proof} Suppose, for a contradiction, there exists a non-trivial linear combination $c_1v_1+...+c_nv_n=0$. Let $m$ be the maximal number with $c_m \ne 0$. Then, $c_1v_1+...+c_mv_m=0$. Therefore, $0=(0,v_{m+1})=(c_1v_1+...+c_mv_m,v_{m+1})=c_1(v_1,v_{m+1})+...+c_m(v_m,v_{m+1})=c_m (v_m,v_{m+1}) \ne 0$, which is a contradiction.
\end{proof}
\begin{lemma} (main):
Let $V$ be an orthogonal, symplectic or unitary space. Let $v_1,...,v_m$ be linearly independent vectors in $V$, which form a chain. Then, $v_1,...,v_m$ can be continued to a maximal chain of $V$.
\end{lemma}
\begin{proof} Let $w_1,...,w_n$ be a maximal chain of $V$. If $n=m$, then the statement is true. Otherwise, $m \le n-1$, so by the small lemma the vectors $w_1,...,w_m$ are linearly independent. Furthermore, let $w_1'=w_1$. Then, we can substitute $w_2$ with $w_2'=c_2w_2$ for some constant $c_2 \ne 0$ so that $(v_1,v_2)=(w_1',w_2')$. Similarly, we can substitute $w_3,...,w_m$ with their scalar multiples so that $(v_i,v_{i+1})=(w_i',w_{i+1}')$ for $i=1,...,m-1$. Let $U=span(v_1,...,v_m)$ and $W=span(w_1,...,w_m)$. Then, $v_1,...,v_m$ form a basis for $U$ and $w_1',...,w_m'$ form a basis for $W$. Hence, we can define a linear map $L:W \rightarrow U$ with $L(w_i')=v_i$ for $i=1,...,m$. Then, by the choice of $w'_1,...,w'_m$, $L$ is an isometry. Therefore, by the Witt's lemma, $L$ can be extended to an isometry $T$ of $V$. Denote, $v_{m+1}=T(w_{m+1})$, ..., $v_n=T(w_n)$. From the fact that $T$ is isometry, one may deduce that $v_1=T(w_1')$, ..., $v_n=T(w_n)$ is a chain as well. 
\end{proof}
It is well known that if $V$ is a symplectic, unitary or orthogonal space, $v \in V$ and $c$ is a constant, then $t(v)=t(cv)$. Therefore, if we have a chain of vectors $v_1,...,v_n$ and multiply some of these vectors by non-zero constants, then it will also be a chain.  
\section{Main results}
Now, after stating and proving important lemmas in the previous sections, we are ready for proving the main results of this paper. We compute the value of $\varphi(G)$ for each group in Fischer's list.
\begin{proposition} If $G \cong S_n$, then $\varphi(G)=n$.
\end{proposition}
The statement is elementary.
\begin{proposition} If $G \cong Sp_{2n}(2)$ and $n \ge 2$, then $\varphi(G)=2n+2$.
\end{proposition}
\begin{proof}
Assume that $V$ is a symplectic space of dimension $2n$ over the $2$-element field. Let $v_1,...,v_{2n}$ be a basis for $V$ with $(v_i,v_j)=1$ if $\{i,j\} \in \{ \{1,2\},\{3,4\},...,$ $\{2n-1,2n\} \}$ and $(v_i,v_j)=0$ otherwise. \newline \newline
Note that the following vectors are linearly independent and form a chain:
\begin{equation*}
    v_1+v_2,v_2+v_3,...,v_{2n-1}+v_{2n},(v_1+v_2+...+v_{2n-2})+v_{2n}.
\end{equation*}
Thus, by the main lemma, this may be continued to a maximal chain. Let $c_1v_1+...+c_{2n}v_{2n}$ be the next vector in the chain. Then, it is orthogonal to each of the vectors
\begin{equation*}
    v_1+v_2,v_2+v_3,...,v_{2n-1}+v_{2n}.
\end{equation*}
Therefore, $c_1=c_2=...=c_{2n}$. Moreover, the vector is non-zero, therefore $c_1=1$. Hence, the next vector in this chain is $v_1+...+v_{2n}$. Note that this chain cannot be continued anymore, as there are no other non-zero vectors orthogonal to vectors $v_1+v_2,v_2+v_3,...,v_{2n-1}+v_{2n}$. Hence, the maximal chain has length $2n+1$, so by the big lemma, $\varphi(G)=(2n+1)+1=2n+2$.
\end{proof}
\begin{proposition} Let $G \cong U_n(2)$. Then, $\varphi(G)=n+1$, when $n \ge 5$ is odd, and $\varphi(G)=n+2$, when $n \ge 4$ is even.
\end{proposition}
\begin{proof}
Assume that $V$ is a unitary space of dimension $n$ over the $4$-element field $\mathbb{F}_4$. Let $v_1,...,v_n$ be an orthonormal basis. We write  $0,1,\alpha,\overline{\alpha}$ for the elements of $\mathbb{F}_4$ and $\overline{x}$ for the conjugate of an element $x$ in $\mathbb{F}_4$. \newline \newline
Now, assume that $n \ge 5$ is odd. Then, the following vectors form a chain and are linearly independent:
\begin{equation*}
    v_1+v_2,v_2+v_3,...,v_{n-1}+v_n.
\end{equation*}
According to the main lemma, this chain can be continued to a maximal chain. Let $c_1v_1+...c_nv_n$ be the next vector in this chain. Then, it is orthogonal to the vectors $v_1+v_2,...,v_{n-2}+v_{n-1}$. Therefore, $c_1=c_2=...=c_{n-1}$. Moreover, this is a singular vector, so $0=(c_1v_1+...+c_nv_n,c_1v_1+...+c_nv_n)=(n-1)c_1\overline{c_1}+c_n \overline{c_n}=c_n \overline{c_n}$, so $c_n=0$. Therefore, up to a constant multiple, there is a unique vector $v_1+...+v_{n-1}$, which continues the chain. Note that $(v_1+v_2)+(v_3+v_4)+...+(v_{n-2}+v_{n-1})+(v_1+...+v_{n-1})=0$, so the vectors in this chain are linearly dependent, therefore by the small lemma, the chain cannot be continued any further. Thus, the maximal chain has length $n$ and $\varphi(G)=n+1$. \newline \newline
Now, assume that $n \ge 4$ is even. Then, the following vectors are linearly independent and form a chain:
\begin{equation*}
    v_1+v_2,v_2+v_3,...,v_{n-1}+v_n,v_1+v_2+...+v_{n-1}+\alpha v_n.
\end{equation*}
According to the main lemma, this chain can be continued to a maximal chain. Let $c_1v_1+...+c_nv_n$ be the next vector in this chain. Then, it is orthogonal to the vectors $v_1+v_2,...,v_{n-1}+v_{n}$. Hence, $c_1=c_2=...=c_{n}$. Therefore, up to a constant multiple, there is a unique vector $v_1+...+v_{n}$, which continues the chain. Note that $(v_1+v_2)+(v_3+v_4)+...+(v_{n-2}+v_{n-1})+(v_1+...+v_{n})=0$, so the vectors in this chain are linearly dependent. Thus, by the small lemma, the chain cannot be continued any further. Thus, the maximal chain has length $n+1$ and $\varphi(G)=(n+1)+1=n+2$.
\end{proof}
\begin{proposition} Let $G \cong PO_n^{+,+}(3)$ and $n \ge 3$. Then, $\varphi(G)=n$, $n \equiv 1$ (mod $3$), and $\varphi(G)=n+1$ otherwise.
\end{proposition}
\begin{proof}
Assume that $(V,Q)$ is an $n$-dimensional orthogonal space over the $3$-element field. Let $v_1,...,v_n$ be a basis for $V$ with $(v_i,v_j)=0$ if $i \ne j$ and $(v_i,v_i)=1$. Note that $Q(v)=2(v,v)=-(v,v)$. \newline \newline
Then, the following vectors are linearly independent and form a chain:
\begin{equation*}
    v_1-v_2,v_2-v_3,...,v_{n-1}-v_n.
\end{equation*}
According to the main lemma, this chain can be continued to a maximal chain. Let $v=c_1v_1+...+c_nv_n$ be the next vector in the chain. Then, it is orthogonal to the vectors $v_1-v_2,v_2-v_3,...,v_{n-2}-v_{n-1}$, so $c_1=c_2=...=c_{n-1}$. Note that if $c_1=0$, then $Q(c_1v_1+...+c_nv_n)=Q(c_nv_n)=c_n^2Q(v_n)=-c_n^2 \ne 1$. So, $c_1 \ne 0$. Note that if $c_1=-1$, then we can multiply $v$ by $-1$. Then, $v=(v_1+...+v_{n-1})+c_nv_n$. \newline \newline
If $n \equiv 0$ (mod $3$), then $1=Q(v)=-(v,v)=-(-1+c_n^2)=1-c_n^2$, thus $c_n^2=0$, which implies $c_n=0$. Thus, $v=v_1+...+v_{n-1}$. Note that $v$ is the only vector (up to a constant multiple) with $Q(v)=1$, which is orthogonal to $v_1-v_2,v_2-v_3,...,v_{n-2}-v_{n-1}$. Thus, this chain cannot be continued anymore. So, the maximal chain has length $n$ and $\varphi(G)=n+1$. \newline \newline
If $n \equiv 1$ (mod $3$), then $1=Q(v)=-(v,v)=-(0+c_n^2)=-c_n^2$, which is impossible. So, the chain $v_1-v_2,...,v_{n-1}-v_n$ is a maximal chain, and $\varphi(G)=(n-1)+1=n$. \newline \newline
If $n \equiv 2$ (mod $3$), then $1=Q(v)=-(v,v)=-(-1+c_n^2)=1-c_n^2$, so $c_n^2=0$, which implies $c_n \in \{ -1, 1 \}$. But $v$ is not orthogonal to $v_{n-1}-v_n$, so $c_n=-1$. Note that, $(v_1+...+v_{n-1}-v_n)-(v_1-v_2)+(v_2-v_3)-(v_4-v_5)+(v_5-v_6)-...-(v_{n-4}-v_{n-3})+(v_{n-3}-v_{n-2})-(v_{n-1}-v_n)=0$. So, the system is linearly dependent, and by the small lemma this chain cannot be continued. Therefore, the maximal chain has length $n$ and $\varphi(G)=n+1$.
\end{proof}
\begin{proposition} Let $G \cong PO_n^{-,+}(3)$ and $n \ge 4$. Then, $\varphi(G)=n$, if $n \equiv 0$ (mod $3$), $\varphi(G)=n+1$, if $n \equiv 2$ (mod $3$) and $\varphi(G)=n+2$, if $n \equiv 1$ (mod $3$).
\end{proposition}
\begin{proof}
Assume that $(V,Q)$ is an $n$-dimensional orthogonal space over the $3$-element field. Let $v_1,...,v_n$ be a basis for $V$ with $(v_i,v_j)=0$ if $i \ne j$, $(v_i,v_i)=1$, if $i<n$ and $(v_n,v_n)=-1$. Note that $Q(v)=2(v,v)=-(v,v)$. \newline \newline
Now, assume that $n \equiv 1$ (mod $3$) with $n \ge 4$. Then, the following vectors form a chain:
\begin{gather*}
    v_1-v_2, v_2-v_3, ..., v_{n-2}-v_{n-1}, (v_1+...+v_{n-2})-v_{n-1}+v_n, \\ (v_1+...+v_{n-2}+v_{n-1})-v_n, v_1+...+v_n.
\end{gather*}
Moreover, as a consequence of the small lemma, there cannot be a chain of length more than $n+1$. Thus, the longest chain has length $n+1$ and $\varphi(G)=(n+1)+1=n+2$. \newline \newline
Now, assume that $n \equiv 2$ (mod $3$) with $n \ge 5$. Then, the following vectors are linearly independent and form a chain:
\begin{equation*}
    v_1-v_2, v_2-v_3, ..., v_{n-2}-v_{n-1}, (v_1+...+v_{n-2})+v_n. 
\end{equation*}
Therefore, this chain can be continued to a maximal chain. Let $v=c_1v_1+...+c_nv_n$ be the next vector. Then, it is orthogonal to vectors $v_1-v_2, v_2-v_3, ..., v_{n-2}-v_{n-1}$. Therefore, $c_1=...=c_{n-1}$. Note that if $c_1 \ne 0$, then $1=Q(v)=(n-1)c_1^2+c_n^2=(n-1)+c_n^2=1+c_n^2$, so $c_n=0$. But then, $(v,(v_1+...+v_{n-2})+v_n)=n-2=0$. So, $c_1=0$. Then, $v=v_n$ (or $-v_n$, these differ by a multiple of $-1$). Note that this chain cannot be continued, because $((v_1+...+v_{n-2})+v_n)-(v_1-v_2)+(v_2-v_3)-(v_4-v_5)+(v_5-v_6)-...-(v_{n-4}-v_{n-3})+(v_{n-3}-v_{n-2})-v_n=0$, so the vectors in this chain are linearly dependent. So, the maximal chain has length $n$ and $\varphi(G)=n+1$. \newline \newline
Finally, assume that $n \equiv 0$ (mod $3$) and $n \ge 6$. Then, the following vectors are linearly independent and form a chain:
\begin{equation*}
    v_1-v_2,v_2-v_3,...,v_{n-2}-v_{n-1}.
\end{equation*}
By the main lemma, this can be continued to a maximal chain. Let $c=c_1v_1+...+c_nv_n$ be the next vector. Then, $c_1=...=c_{n-2}$. Then $1=Q(v)=-(v,v)=-(1+c_{n-1}^2-c_n^2)$. The only possibility is $c_n=0$ and $c_{n-1} \ne 0$. But then $(v,v_{n-2}-v_{n-1}) \ne 0$, so $c_{n-1}=-1$. So, $v=(v_1+...+v_{n-2})-v_{n-1}$. Note that this chain cannot be continued, because $((v_1+...+v_{n-2})-v_{n-1})-(v_1-v_2)+(v_2-v_3)-(v_4-v_5)+(v_5-v_6)-...+(v_{n-4}-v_{n-3})-(v_{n-2}-v_{n-1})=0$, is linearly dependent. Thus, the maximal chain has length $n-1$ and $\varphi(G) = (n-1)+1=n$.
\end{proof}
In \cite{As96} it is proven that if $n$ is odd, then $PO_n^{+,-}(3) \cong PO_n^{-,+}(3)$ and $PO_n^{-,-}(3) \cong PO_n^{+,+}(3)$. Moreover, if $n$ is even, then $PO_n^{+,-}(3) \cong PO_n^{+,+}(3)$ and $PO_n^{-,-}(3) \cong PO_n^{-,+}(3)$. Therefore, the previous propositions can be used to find the results for the groups $PO_n^{\mu,-}(3)$ as well. Moreover, if $n$ is odd, then $O_n^{+}(2) \cong O_n^{-}(2) \cong Sp_{n-1}(2)$ as a $3$-transposition group. 
\begin{proposition} Let $G \cong O_n^{+}(2)$, $n \ge 6$ be even. Then,  $\varphi(G)=n$, if $n \equiv 2$ (mod $8$), $\varphi(G)=n+2$, if $n \equiv 6$ (mod $8$) and $\varphi(G)=n+1$, if $4 | n$ and $n \ge 8$.
\end{proposition}
\begin{proof}
Assume that $V$ is an $2n$-dimensional orthogonal space of $+$ type. Let $v_1,...,v_n$ be a basis with $Q(v_1)=...=Q(v_{n-1})=1$ and $Q(v_n)=1$ if and only if $4 | n$. Also, $(v_1,v_2)=(v_3,v_4)=...=(v_{n-1},v_n)=1$ and the other products are $0$. \newline \newline
First, assume $8 | n$. Then the following vectors are linearly independent and form a chain:
\begin{gather*}
    v_1+v_2,v_2+v_3+v_{n-1},v_3+v_4,v_4+v_5+v_{n-1},...,v_{n-3}+v_{n-2}, \\
    (v_1+...+v_{n-4})+v_{n-2}, v_1+v_2+v_5+v_6+...+v_{n-3}+v_{n-2}+v_n,v_{n-1}.
\end{gather*}
By the main lemma, this chain can be continued to a maximal chain. Note that if the set $w_1,...,w_n$ is a basis for $V$, then the system of equations $(w,w_1)=...=(w,w_{n-1})=0$, $(w,w_n)=1$ has a unique solution. Note that $w_1+w_3+...+w_{n-1}$ is a solution for this system. However, $Q(w_1+w_3...+w_{n-1})=\frac{n}{2}=0$ is singular. Thus, the maximal chain has length $n$ and $\varphi(G)=n+1$. \newline \newline
Now, assume $n \equiv 2$ (mod $8$). Then the following vectors are linearly independent and form a chain:
\begin{gather*}
    v_1+v_2,v_2+v_3+v_{n-1},v_3+v_4,v_4+v_5+v_{n-1},...,v_{n-3}+v_{n-2}, \\
    (v_1+...+v_{n-4})+v_{n-2}+v_{n-1}, v_1+v_2+...+v_{n-1}.
\end{gather*}
By the main equation this chain can be continued to a maximal chain. Let $v=c_1v_1+...+c_nv_n$ be the next vector in the chain. Then, because of conditions of orthogonality, $c_1=c_2$, $c_3=c_4$, ..., $c_{n-3}=c_{n-2}$. In addition, $c_1+c_3+c_n=c_3+c_5+c_n=...=c_{n-5}+c_{n-3}+c_n=0$. In addition, $c_{n-3}+c_n=0$ and $c_n=1$. So, $v=(v_3+v_4)+(v_7+v_8)+...+(v_{n-3}+v_{n-2})+(cv_{n-1}+v_n)$, which is singular. Thus, this is a maximal chain and $\varphi(G)=(n-1)+1=n$. \newline \newline
First, assume $n \equiv 4$ (mod $8$). Then the following vectors are linearly independent and form a chain:
\begin{gather*}
    v_1+v_2,v_2+v_3+v_{n-1},v_3+v_4, \\
    v_4+v_5+v_{n-1},...,v_{n-3}+v_{n-2}, (v_1+...+v_{n-4})+v_{n-2},\\
    v_1+v_2+...+v_{n-2},(v_3+v_4)+(v_7+v_8)+...+(v_{n-5}+v_{n-4})+v_n.
\end{gather*}
By the main lemma, this chain can be continued to a maximal chain. Note that if the set $w_1,...,w_n$ is a basis for $V$, then the system of equations $(w,w_1)=...=(w,w_{n-1})=0$, $(w,w_n)=1$ has a unique solution. Note that $w_1+w_3+...+w_{n-1}$ is a solution for this system. However, $Q(w_1+w_3...+w_{n-1})=\frac{n}{2}=0$ is singular. Thus, the maximal chain has length $n$ and $\varphi(G)=n+1$. \newline \newline
First, assume $n \equiv 6$ (mod $8$). Then the following vectors are linearly independent and form a chain:
\begin{gather*}
    v_1+v_2,v_2+v_3+v_{n-1},v_3+v_4,v_4+v_5+v_{n-1},...,v_{n-3}+v_{n-2}, \\
    (v_1+...+v_{n-4})+v_{n-2}+v_{n-1},\\
    v_1+v_2+...+v_{n-1},(v_3+v_4)+(v_7+v_8)+...+(v_{n-5}+v_{n-4})+v_n, v_{n-1}.
\end{gather*}
Note that this is a maximal chain, because by the small lemma, the chain has length at most $n+1$. Therefore, $\varphi(G)=(n+1)+1=n+2$.
\end{proof}
\begin{proposition} Let $G \cong O_n^{-}(2)$, $n \ge 4$ be even. Then,   $\varphi(G)=n$, if $n \equiv 6$ (mod $8$), $\varphi(G)=n+2$, if $n \equiv 2$ (mod $8$), $\varphi(G)=n+1$, if $4 | n$.
\end{proposition}
The proof of this proposition is analogical to the proof of the previous proposition.
\begin{proposition}
For the small groups we have the following results:
\begin{itemize}
    \item $\varphi(Sp_2(2))=3$.
    \item $\varphi(U_1(2))=1$.
    \item $\varphi(U_2(2))=3$.
    \item $\varphi(U_3(2))=3$.
    \item $\varphi(PO^{+,+}_1(3))=1$.
    \item $\varphi(PO^{+,+}_2(3))=2$.
    \item $\varphi(PO^{-,+}_1(3))=2$.
    \item $\varphi(PO^{-,+}_2(3))=2$.
    \item $\varphi(PO^{-,+}_3(3))=2$.
    \item $\varphi(O^{+}_2(2))=2$.
    \item $\varphi(O^{+}_4(2))=3$.
    \item $\varphi(O^{-}_2(2))=3$.
\end{itemize}
\end{proposition}
\begin{proof}
These results can be verified elementary using the technique of the proof of the above propositions.
\end{proof}
\begin{proposition} Let $G$ be of type $M(22)$. Then $\varphi(G)=10$.
\end{proposition}
\begin{proof}
Let $x_1,...,x_n$ be a chain of transpositions of $G$. Then, $x_1,...,x_{n-2}$ is a chain of transpositions for $C_{x_n}(G)/<x_n> \cong U_6(2)$. Therefore, by the previous propositions, $n-2 \le 7$ and $n \le 9$. \newline \newline
Now, let us construct a chain of length $9$. Let us pick a pair of non-commuting transpositions $x_1,x_2$. Then, $C_{x_1,x_2}(G) \cong PO_6^{+,+}(3)$. Therefore, there exists a chain of involutions $x_4,...,x_9$ commuting with $x_1$ and $x_2$. Now, let $f:C_{x_1}(G)/<x_1> \rightarrow U_6(2)$ be an isomorphism. Then, $f(x_i)=t(v_i)$, where $v_i$ are non-zero singular vectors. Let $v_3=v_5+v_7+v_9$. Then, $v_3 \ne 0$ is singular. Denote, $x_3=f^{-1}(t(v_i))$. Note that $x_3,...,x_9$ is a chain in $C_{x_1}(G)$, but not in $C_{x_1,x_2}(G)$, because the maximal chain there has length $6$. So, $x_2$ and $x_3$ do not commute. Hence $x_1,...,x_9$ is a chain and it is maximal. Therefore, $\varphi(G)=9+1=10$.
\end{proof}
\begin{proposition} Let $G$ be of type $M(23)$. Then, $\varphi(G)=12$.
\end{proposition}
\begin{proof}
Let $x_1,...,x_n$ be a chain of transpositions of $G$. Then, $x_1,...,x_{n-2}$ is a chain of transpositions for $C_{x_n}(G)/<x_n>$, which is of type $M(22)$. Therefore, by the previous proposition, $n-2 \le 9$ and $n \le 11$. Furthermore, let us prove the following lemma.
\begin{lemma} Let $H$ be a group of $M(22)$ type and $g_1,...,g_8$ be a chain of transpositions. Then, there exists a unique transposition $g_9$ so that $g_1,...,g_9$ is a chain.
\end{lemma}
\begin{proof} First of all, as we can see from the proof of the previous part if such an element $g_9$ exists, then $g_3,...,g_9 \in C_H(g_1)$, so $g_9$ is defined uniquely. Therefore, it suffices to prove that the number of the chains of length $8$ and the number of chains of length $9$ are equal. Let $K \cong PO_6^{+,+}(3)$ and for $1 \le n \le 6$ and let $v_1,...,v_{n-1}$ form a chain. Let $k_n$ be the number of vectors $v_n$ so that $v_1,...,v_n$ is a linearly independent chain. By the main lemma, this number is independent of the choice of $v_1,...,v_n$. \newline \newline
Now, let us compute the number of $9$-chains: there are $3510$ choices for $g_1$ and then there are $2816$ choices for $g_2$. Then, $C_{g_1,g_2} \cong PO_6^{+,+}(3)$. Thus, there are $k_1$ choices for $g_9$, $k_2$ choices for $g_8$, ..., $k_6$ choices for $g_4$. Finally, from the proof of the previous part, $g_1,g_2,g_4,g_5,...,g_9$ define $g_3$ in a unique way. Thus, the number of the $9$-chains is $3510\cdot2816\cdot k_1 \cdot ... \cdot k_6$. Now, let us compute the number of $8$-chains. Again, there are $3510\cdot2816$ choices for the ordered pair $(g_1,g_2)$, $k_1$ choices for $g_8$, ..., $k_5$ choices for $g_4$. Note that $g_4, ..., g_8$ is a chain in $C_{g_1}(H)$. Let $f:C_{g_1}(H)/<g_1> \rightarrow SU_6(2)$ be an isomorphism and $f(g_i)=t(v_i)$. Then, the system of equations of $(v_3,v_4)=1$, $(v_3,v_5)=...=(v_3,v_8)=0$ has $2$ solutions (by Witt's lemma, it is sufficient to verify this argument just for one example - pick $v_4=(110000)$, $v_5=(011000)$, ..., $v_8=(000011)$ with respect to the unitary basis, then there are two such vectors $(\overline{\alpha} \alpha \alpha \alpha \alpha \alpha)$ and $(\alpha \overline{\alpha} \overline{\alpha} \overline{\alpha} \overline{\alpha} \overline{\alpha})$). Hence, there are $2$ such involutions $g_3=f^{-1}(t(v_3))$, $k_6$ of which are in $C_{g_1,g_2}(H)$. Therefore, the number of $8$-chains is $3510\cdot2816\cdot k_1 \cdot ... \cdot k_5 \cdot (2-k_6)$. Thus, it is sufficient to verify that $3510\cdot2816\cdot k_1 \cdot ... \cdot k_6 = 3510\cdot2816\cdot k_1 \cdot ... \cdot k_5 \cdot (2-k_6)$ or equivalently $k_6=1$. $k_6$ does not depend on the choice $v_1,...,v_5$ so pick $v_1=(120000)$, $v_2=(012000)$, ..., $v_5=(000012)$ with respect to the orthogonal basis. Then, if $v_6=(c_1c_2c_3c_4c_5c_6)$, then $c_1=c_2=...=c_5$. Moreover, $Q(v_6)=1$, so $c_6=0$. Hence, either $v_6=(111110)$ or $v_6=(222220)$, so $g_6=t(111110)=t(222220)$ is unique! 
\end{proof}
Now, let us construct a chain of length $11$. Let us pick a pair of non-commuting transpositions $x_1,x_2$. Then, $C_{x_1,x_2}(G) \cong PO_7^{+,-}(3)$. Therefore, there exists a chain of involutions $x_4,...,x_{11}$ commuting $x_1$ and $x_2$. By the lemma, we can continue the chain $x_{11},...,x_4$ of $C_{x_1}(G)/<x_1>$ with $x_3$. Note that $x_3,...,x_{11}$ is a chain in $C_{x_1}(G)$, but not in $C_{x_1,x_2}(G)$, because the maximal chain there has length $8$. So, $x_2$ and $x_3$ do not commute. Hence $x_1,...,x_{11}$ is a chain and it is maximal. Therefore, $\varphi(G)=11+1=12$.
\end{proof}
\begin{proposition} Let $G$ be of type $M(24)$. Then, $\varphi(G)=12$.
\end{proposition}
\begin{proof}
Firstly, note that $G$ has a subgroup of type $M(23)$, so $\varphi(G) \ge 12$. Furthermore, if $g_1,...,g_n$ is a chain of transpositions, then $g_4,...,g_n \in C_{g_1,g_2}(G)$, which is isomorphic to a triple cover of $PO^{+,+}_8(3)$. By the previous propositions, the maximal chain chain there has length $8$. So, $n-3 \le 8$ and $n \le 11$. So, $\varphi(G) \le 11+1=12$. Thus, $\varphi(G)=12$, as desired.
\end{proof}
\section{Geometrical interpretation of the main results}
In this section, we discuss the inverse of the problem in the previous section for the classical groups. We build representations of symmetric groups $S_n$ in the spaces with symplectic, unitary, or orthogonal forms. From the previous part, one may easily deduce, that if $G$ is a symplectic, unitary or orthogonal group of dimension $n$, then $n \le \varphi(G) \le n+2$. Therefore, we anticipate that $S_n$ should act on a space of dimension $m$, with $n-2 \le m \le n$. For that purpose, we build a symplectic, unitary, or orthogonal space $V$ with basis $v_1,...,v_n$ and allow $S_n$ to permute the basis. Then, we take the subspace, which is orthogonal to some vector, and, if possible, factorize it over some vector. If needed, for fixing the sign of an orthogonal space, we take a direct sum with a one-dimensional space. \newline \newline
In this section we look for a representation of $S_n$ with $n \ge 5$. Actually, in this case, the action of $S_n$ on the corresponding space is faithful, because the kernel of the action is a normal subgroup of $S_n$, that is $\{ 1 \}, A_n$ or $S_n$. And the image of the representation has more than two elements, so the kernel is the trivial subgroup. 
\subsection{Symplectic spaces over the 2-element field} 
Let $n$ be an even integer. Suppose that $W$ is an $n$-dimensional symplectic space over the $2$-element field. Let $w_1 ,..., w_n$ be a basis for $W$ with $(w_i, w_j)=1$, if $i \ne j$. Let $W' =\{x \in W : (x, w_1+...+w_n)=0\}$. Then, $W'$ is a degenerate space : $w_1+...+w_n \in W^{'}$. Thus, if we factor $W^{'}$ by the subspace generated by the vector $w_1+...+w_n$ we obtain the space $V=W^{'}/<w_1+...+w_n>$ which is a symplectic space of dimension $n-2$. Let $\tau \in S_n$. Then, let $\psi(\tau):W\rightarrow W$ be a linear operator for which $\psi(\tau)(w_i)=w_{\tau(i)}$. Furthermore, one may restrict $\psi(\tau)$ on $V$. Note that $\psi(\tau)$ preserves the symplectic form on $V$, so $\psi$ is the natural embedding of $S_n$ into $Sp_{n-2}(2)$. Finally, $\psi(i j)(w_i)=w_j$, $\psi(i j)(w_j)=w_i$, and $\psi(i j)(w_k)=w_k$ if $k \ne i,j$ in $W$. Note that the transvention $t(w_i+w_j)$ defined on $W$ acts in the same way on the basis $(v_1,...,v_n)$. On the induced space $V$, $t(w_i+w_j)$ is also a transvection, as $w_i+w_j \in W'$. Thus, $\psi(i j)=t(w_i+w_j)$ on $V$. Therefore, we verified that the images of transpositions under $\tau$ are transvenctions of $V$. 
\subsection{Unitary spaces over the 4-element field} 
Let $n$ be an even integer. Suppose that $W$ is an $n$-dimensional unitary space over the $4$-element field. Let $w_1 ,..., w_n$ be a basis for $W$ with $(w_i,w_i)=1$ and $(w_i, w_j)=0$, if $i \ne j$. Let $W' =\{x \in W : (x, w_1+...+w_n)=0\}$. Then, $W'$ is a degenerate space : $w_1+...+w_n \in W^{'}$. Thus, if we factor $W^{'}$ by the subspace generated by the vector $w_1+...+w_n$ we obtain the space $V=W^{'}/<w_1+...+w_n>$ which is a unitary space of dimension $n-2$. Let $\tau \in S_n$. Then, let $\psi(\tau):W\rightarrow W$ be a linear operator for which $\psi(\tau)(w_i)=w_{\tau(i)}$. Furthermore, one may restrict $\psi(\tau)$ on $V$. Note that $\psi(\tau)$ preserves the unitary form on $V$, so $\psi$ is the natural embedding of  $S_n$ into $U_{n-2}(2)$. Finally, $\psi(i j)(w_i)=w_j$, $\psi(i j)(w_j)=w_i$, and $\psi(i j)(w_k)=w_k$ if $k \ne i,j$ in $W$. Note that the transvention $t(w_i+w_j)$ defined on $W$ acts in the same way on the basis $(w_1,...,w_n)$. On the induced space $V$, $t(w_i+w_j)$ is also a transvection, as $w_i+w_j \in W'$. Thus, $\psi(i j)=t(w_i+w_j)$ on $V$ and $w_i+w_j$ is singular. Therefore, we verified that the images of transpositions under $\tau$ are transvenctions of $V$.
\subsection{Orthogonal spaces over the $3$-element field} 
Let $n$ be an integer congruent to $0$ or $1$ modulo $3$. Suppose that $W$ is an $(n+1)$-dimensional orthogonal space over the $3$-element field of type "-". Let $w_1 ,..., w_n,w_{n+1}$ be a basis for $W$ with $Q(w_i)=-1$, $(w_i,w_i)=1$, for $i=\overline{1,...,n}$, $Q(w_{n+1})=1$, $(w_{n+1},w_{n+1})=-1$  and $(w_i, w_j)=0$, if $i \ne j$. Let $v=w_1+...+w_n$ if $n \equiv 0$ (mod $3$) and $v=w_1+...w_{n+1}$ if $n \equiv 1$ (mod $3$). Let $W' =\{x \in W : (x, v)=0\}$. Then, $W'$ is a degenerate space : $v \in W^{'}$. Thus, if we factor $W^{'}$ by the subspace generated by the vector $v$, we obtain the space $V=W^{'}/<v>$ which is a unitary space of dimension $n-1$. Let $\tau \in S_n$. Then, let $\psi(\tau):W\rightarrow W$ be a linear operator for which $\psi(\tau)(w_i)=w_{\tau(i)}$. Furthermore, one may restrict $\psi(\tau)$ on $V$. Note that $\psi(\tau)$ preserves the orthogonal form on $V$. \newline \newline 
If, $n \equiv 1$ (mod $3$), then the images of the vectors $w_1-w_2, w_1-w_3, ..., w_1-w_n$ form a basis for $V$. Let us construct the Gram matrix: note that if $i\neq j$, then $(w_1-w_{i}, w_1-w_{j})=-1$ and $(w_1-w_{i}, w_1-w_{j})=1$. 
So, the Gram matrix is of form 
\[
A = 
 \begin{pmatrix}
  -1 & 1 & \cdots & 1 \\
  1 & -1 & \cdots & 1 \\
  \vdots  & \vdots  & \ddots & \vdots  \\
  1 & 1 & \cdots & -1
 \end{pmatrix}
\]
Note that, $(A-I)^2=0$, So $A-I$ has all eigenvalues $0$. Hence, all the eigenvalues of $A$ are $1$ and $det(A)=1$. So, $V$ is of type $"+"$. \newline \newline
Furthermore, $n \equiv 0$ (mod $3$), then the images of the vectors $w_1-w_2, w_1-w_3, ..., w_1-w_{n-1},w_{n+1}$ form a basis for $V$. Let us construct the Gram matrix: note that if $i\neq j$, then $(w_1-w_{i}, w_1-w_{j})=-1$ and $(w_1-w_{i}, w_1-w_{j})=1$. Also, $(w_{n+1},w_{n+1})=-1$ and $w_{n+1}$ is orthogonal to all other vectors.
So, the Gram matrix is of form 
\[
A = 
 \begin{pmatrix}
  -1 & 1 & \cdots & 1 & 0 \\
  1 & -1 & \cdots & 1  & 0\\
  \vdots  & \vdots  & \ddots & \vdots  \\
  1 & 1 & \cdots & -1 & 0 \\
  0 & 0 & \cdots & 0 & -1
 \end{pmatrix}
\]
Note that, $rank(A-I)=2$, So $A-I$ has all eigenvalues, but $2$ eigenvalues equal $0$. Furthermore, it has $2$ vectors corresponding to the eigenvalue $1$. So, $Sp(A-I)=\{0,0,...,0,1,1\}$. Thus, $Sp(A)=\{1,1,...,1,-1,-1\}$ and $det(A)=1$. Therefore, $V$ is of type $"+"$. \newline \newline
Finally, if $(i j) \in S_n$, then $\psi(i j)=t(w_i-w_j)$ and $Q(w_i-w_j)=Q(w_i)+Q(w_j)=1$. Thus, $\psi$ is the natural embedding of $S_n$ into $O^{+,+}_{n-1}(3)$. \newline \newline
Let $n$ be a multiple of $3$. Suppose that $W$ is an $n$-dimensional orthogonal space over the $3$-element field of type "+". Let $w_1 ,..., w_n$ be a basis for $W$ with $Q(w_i)=-1$, $(w_i,w_i)=1$, and $(w_i, w_j)=0$, if $i \ne j$. Let $W' =\{x \in W : (x, w_1+...+w_n)=0\}$. Then, $W'$ is a degenerate space : $w_1+...+w_n \in W^{'}$. Thus, if we factor $W^{'}$ by the subspace generated by the vector $w_1+...+w_n$, we obtain the space $V=W^{'}/<w_1+...+w_n>$ which is an orthogonal space of dimension $n-2$. Let $\tau \in S_n$. Then, let $\psi(\tau):W\rightarrow W$ be a linear operator for which $\psi(\tau)(w_i)=w_{\tau(i)}$. Furthermore, one may restrict $\psi(\tau)$ on $V$. Note that if $(i j) \in S_n$, then $\psi(i j)=t(v_i-v_j)$ and $Q(v_i-v_j)=1$. Moreover, calculations, similar to the calculations carried out in the previous subcase show that the space $V$ is of type "-". Thus, $\psi$ is the natural embedding of $S_n$ into $PO^{-,+}_{n-2}(3)$. 
\subsection{Orthogonal spaces over the 2-element field} 
Let $n$ be a multiple of $4$. Let $W$ be a symplectic space over the $2$-element field with basis $w_1,...,w_n$ such that $(w_i,w_j)=1$ if $i \ne j$. Equip $W$ with an orthogonal form such that $Q(w_1)=...=Q(w_n)=0$. Let $W' =\{x \in W : (x, w_1+...+w_n)=0\}$. Then, $W'$ is a degenerate space : $w_1+...+w_n \in W^{'}$. Moreover, $Q(w_1+w_2+...+w_n)=\frac{n(n-1)}{2} = 0 \in \mathbb{Z}_2$. Thus, if we can factor $W^{'}$ by the subspace generated by the vector $w_1+...+w_n$ and obtain the space $V=W^{'}/<w_1+...+w_n>$ which is an orthogonal space of dimension $n-2$. Let $\tau \in S_n$. Then, let $\psi(\tau):W\rightarrow W$ be a linear operator for which $\psi(\tau)(w_i)=w_{\tau(i)}$. Furthermore, one may restrict $\psi(\tau)$ on $V$. Note that if $(i j) \in S_n$, then $\psi(i j)=t(v_i+v_j)$ and $Q(v_i+v_j)=1$. Thus, $\psi$ is the natural embedding of $S_n$ into $O^{\epsilon}_{n-2}(2)$ Also, from the previous section, one may easily deduce, that $\epsilon=-$, if $n \equiv 4$ (mod $8$) and the subspace is of type $\epsilon=+$ if $8|n$. \newline \newline
Let $n$ be an integer congruent to $2$ modulo $4$. Let $W$ be a symplectic space over the $2$-element field with basis $w_1,...,w_n$ such that $(w_i,w_j)=1$ if $i \ne j$. Equip $W$ with an orthogonal form such that $Q(w_1)=...=Q(w_n)=\epsilon$. Note that $W=<w_1,w_2> \oplus <w_1+w_2+w_3,w_1+w_2+w_4> \oplus ... \oplus <w_1+w_2+...+w_{n-2}+w_{n-1},w_1+w_2+...+w_{n-2}+w_{n}>$. Moreover, types of the spaces alternate. Thus, the spaces with $\epsilon=0$ and $\epsilon=1$ are not isomorphic. Furthermore, $S_n$ acts coordinate-wise on both spaces, and a transposition $(i j)$ on $S_n$ has the same action as the transvenction $t(v_i+v_j)$. Therefore, we described a faithful action of $S_n$ on both $O_n^{+}(2)$ and $O_n^{-}(2)$. \newline \newline
Let $n$ be an integer congruent to $1$ modulo $4$. Suppose that $W$ is an orthogonal space of dimension $(n+1)$. Let $v_1,...,v_{n+1}$ be a basis for $W$ with $(v_i, v_j)=1$ if $v_i \neq v_j$ $Q(v_1)=...=Q(v_{n+1})=0$. The choice of $Q(v_{n+2})$ depends on the type of the orthogonal space.  Note that $Q(v_1+...v_{n+1})=\frac{(n+1)n}{4}=0 \in Z_2$. Thus, if $W^{'}=\{w_1+w_2+...+w_{n+1}\}$ and $V=W^{'}/<w_1+...+w_{n+1}>$, then $V$ is an orthogonal space, which inherits the dot product and quadratic form from $W$. Let $\tau \in S_n$. Then, let $\psi(\tau):W\rightarrow W$ be a linear operator for which $\psi(\tau)(w_i)=w_{\tau(i)}$, for $i=\overline{1,...,n+1}$. Moreover, one may restrict $\psi(\tau)$ on $V$. Note that $\psi(\tau)$ preserves the orthogonal form, and for $i\ne j \in \{1,...,n\}$ $\psi(i j)=t(v_i+v_j)$. Thus, $\psi$ is the natural embedding of $S_n$ into $O_{n-1}^{\epsilon}(2)$.
\section{Conclusions} To conclude the results obtained in the previous sections, we have computed the following values for the groups in Fischer's list.
\begin{theorem} The maximal symmetric subgroups of the groups from the Fischer's group have the following sizes:
\begin{itemize}
    \item $\varphi(S_n)=n$;
    \item $\varphi(Sp_2(2))=3$, $\varphi(Sp_n(2))=n+2$, if $n \ge 4$ is even;
    \item $\varphi(U_n(2))=n$, if $n \le 3$, $\varphi(U_n(2))=n+1$, if $n \ge 4$ is odd, $\varphi(U_n(2))=n+2$, if $n \ge 4$ is even;
    \item $\varphi({PO_n}^{+,+}(3))=n$, if $n=1$ (mod $3$), $n+1$, otherwise;
    \item $\varphi({PO_n}^{-,+}(3))=n+2$, if $n=1$ (mod $3$), $n+1$ if $n=0$ (mod $3$), $n$, if $n=2$ (mod $3$);
    \item $\varphi({O_4}^{+}(2))=3$, $\varphi({O_n}^{+}(2))=n+1$, if $n=0$ (mod $4$), $\varphi({O_n}^{+}(2))=n$, if $n=2$ (mod $8$) and $\varphi({O_n}^{+}(2))=n+2$, if $n=6$ (mod $8$);
    \item $\varphi({O_2}^{-}(2))=3$, $\varphi({O_n}^{+}(2))=n+1$, if $n=0$ (mod $4$), $\varphi({O_n}^{+}(2))=n$, if $n=6$ (mod $8$) and $\varphi({O_n}^{+}(2))=n+2$, if $n=2$ (mod $8$);
    \item $\varphi(M(22))=10$, $\varphi(M(23))=12$, $\varphi(M(24))=12$.
\end{itemize}
\end{theorem}
In addition, in the previous section, for each classical group $G$ we have built a faithful representation of $S_{\varphi(G)}$ into the corresponding space of $G$, or into a smaller space, which is a direct geometrical interpretation of the main results. \newline \newline
The theorem \ref{main} is a direct corollary of this theorem. The groups listed in the theorem \ref{main} are candidates to admit standard Majorana representations. Note that if for such a group $(G,D)$ there exists a monomorphism $\psi:G \rightarrow M$, such that $\psi(D) \subseteq 2A$, then there is also a Majorana representation of $G$ - it suffices to take the subalgebra of the algebra $V_M$ generated by the axes corresponding to the involutions from $\psi(D)$. Therefore, we plan to find out the groups in this list, which admit a standard Majorana representation, and among these groups find these which can be embedded in the Monster.
\section{Acknowledgments}
I would like to thank my supervisor, A.A. Ivanov, for introducing me to the area of research and for his useful advice.

\end{document}